\documentclass[12pt]{amsart}

\usepackage[a4paper,centering,margin=2.5cm]{geometry}
\usepackage{cite}
\usepackage{color,hyperref,cleveref}
\definecolor{darkblue}{rgb}{0.0,0.0,0.3}
\hypersetup{colorlinks,breaklinks,
            linkcolor=darkblue,urlcolor=darkblue,
            anchorcolor=darkblue,citecolor=darkblue}

\usepackage{amssymb}
\usepackage{graphicx,psfrag}
\def\boxit{$\sqcap\kern-8pt\sqcup$}

\newcommand{\R}{\mathbb{R}}

\theoremstyle{plain}
\newtheorem{theorem}{Theorem}[section]
\newtheorem{lemma}[theorem]{Lemma}
\newtheorem{proposition}[theorem]{Proposition}
\newtheorem{corollary}[theorem]{Corollary}

\newtheorem{remark}[theorem]{Remark}

\theoremstyle{definition}

\newtheorem{conjecture}[theorem]{Conjecture}
\newtheorem*{conjecture*}{Conjecture}

\newtheorem*{problem*}{Problem}

\newcommand{\M}{\mathcal{M}}

\title{New results on $k$-Roudneff's conjecture}
\dedicatory{}

\author{Rangel Hern\'andez-Ortiz}
\address{Universitat Rovira i Virgili, Departament d'Enginyeria Inform\`{a}tica i Matem\`{a}tiques, Av. Pa\"{i}sos Catalans 26, 43007 Tarragona, Spain}
\email{rangel.hernandez@urv.cat}

\author{Luis Pedro Montejano}
\address{Universitat Rovira i Virgili, Departament d'Enginyeria Inform\`{a}tica i Matem\`{a}tiques, Av. Pa\"{i}sos Catalans 26, 43007 Tarragona, Spain}
\email{luispedro.montejano@urv.cat}


\date{\today}

\begin{document}

\begin{abstract}
In this paper we study the number of $k$-neighborly reorientations of  an oriented matroid, leading to study $k$-Roudneff’s conjecture, the case $k = 1$ being the original statement conjectured in 1991. We first prove the conjecture for the family of Lawrence oriented matroids (LOMs) with even rank $r=2k+2$ and also for low ranks by computer. Next, we provide a general upper bound for the number of $k$-neighborly reorientations of any LOM. Finally, we prove that for any $k\ge1$ and any 
oriented matroid on $n$ elements, $k$-Roudneff's conjecture holds asymptotically as $n\rightarrow \infty$ and thus giving more credit to the conjecture.


\end{abstract}

\maketitle

\section{Introduction}\label{intro}

In this paper we study the number of $k$-neighborly reorientations of a rank $r$ oriented matroid $\M$ on $n$ elements, denoted by $f_{\M}(r,n,k)$. We obtain some partial results on $k$-Roudneff's conjecture (proposed in \cite{HKM24}), the case $k=1$ being a well-known conjecture stated originally in terms of projective pseudohyperplane arrangements \cite{R91}.
\smallskip

Let us briefly give some basic notions and definitions on oriented matroid theory needed for the rest of the paper. We refer the reader to \cite{BVSWZ99} for background on oriented matroid theory.
 Given an \emph{oriented matroid}  $\M=(E,\mathcal{C})$ of a finite \emph{ground set} $E$ and a set of \emph{sign-vectors} $\mathcal{C}\subseteq \{+,-,0\}^E$ called \emph{circuits}, the \emph{size} of a member  $X\in\{+,-,0\}^E$ is the size of its \emph{support} $\underline{X}=\{e\in E\mid X_e\neq 0\}$. Throughout the paper all oriented matroids are considered \emph{simple}, i.e., all circuits have size at least $3$. The \emph{rank} $r$ of $\M$ is the size of a largest set not containing a circuit of $\M$.  An oriented matroid of rank $r$ is called \emph{uniform} if all its circuits are of size $r+1$.
For a \emph{sign-vector} $X\in\{+,-,0\}^E$ on ground set $E$ we denote by $X^+:=\{e\in E\mid X_e=+\}$ the set of \emph{positive elements} of $X$ and $X^-:=\{e\in E\mid X_e=-\}$ its set of
\emph{negative elements}. Hence, the set $\underline{X}=X^+\cup X^-$ is the support of $X$. For a subset $R\subseteq E$ the \emph{reorientation of $R$} is the oriented matroid $_{-R}\M$ obtained from $\M$ by reversing the sign of $X_e$ for every $e\in R$ and $X\in\mathcal{C}$. The set of all oriented matroids that can be obtained this way from $\M$ is the \emph{reorientation class} $[\M]$ of $\M$.
 We denote by $-X$ the sign-vector $_{-E}X$ where all signs are reversed, i.e., such that $-X^+=X^-$ and $-X^-=X^+$.
 We say that $X$ is \emph{positive} if $X^-=\emptyset$ and $\underline{X}\neq\emptyset$.

\smallskip

Given an oriented matroid $\M=(E,\mathcal{C})$ the \emph{contraction} of one element $e\in E$ is the oriented matroid $\mathcal{M}/e=(E\setminus \{e\} ,\mathcal{C}/e)$, where $\mathcal{C}/e$ is the set of support-minimal sign-vectors from $\{X\setminus e\mid X\in\mathcal{C}\}\setminus\{0\}$ where $X\setminus e$ is the sign-vector on groundset $E\setminus \{e\}$ such that $(X\setminus e)_f=X_f$ for all $f\in E\setminus \{e\}$. If $\M$ is a uniform oriented matroid of rank $r$, then $\mathcal{M}/e$ is uniform of rank $\max\{0,r-1\}$.
The \emph{deletion} of $e$ from $\mathcal{M}$ is the oriented matroid $\mathcal{M}\setminus e=(E\setminus \{e\} ,\mathcal{C}\setminus e)$, where $\mathcal{C}\setminus e=\{X\setminus e\mid X\in\mathcal{C}, \underline{X}\cap \{e\}=\emptyset\}\setminus\{0\}$. If $\M$ is a uniform oriented matroid of rank $r$, then $\mathcal{M}/e$ is uniform of rank $\min\{|E\setminus \{e\}|,r\}$.

\smallskip
An oriented matroid $\mathcal{M}=(E,\mathcal{C})$ is \emph{acyclic} if every circuit has positive and negative signs, i.e., if $|X^+|>0$ and $|X^-|>0$ for every circuit $X\in \mathcal{C}$. Given an integer $k\ge 0$, we say that $\mathcal{M}$ is    \emph{$k$-neighborly} if  $|X^+|>k$ and $|X^-|>k$ for every $X\in \mathcal{C}$, in other words, if for every subset $S\subseteq E$ of size at most $k$ the reorientation $_{-S}\M$ is acyclic.
A reorientation $R$ of $\mathcal{M}$ is a  \emph{$k$-neighborly reorientation} if $_{-R}\M$ is $k$-neighborly.
It follows from the definition that a $k$-neighborly oriented matroid is  $k'$-neighborly for all $0\leq k'\leq k$.
Note that $\M$ is $0$-neighborly if and only if $\M$ is acyclic. If $\M$ is $1$-neighborly, then $\M$ is called \emph{matroid polytope}.
If $\M$ is a rank $r$, $k$-neighborly  oriented matroid, then $r\ge 2k+1$. If $r=2k+1$ then $\M$ is often just called \emph{neighborly} (i.e., when $k$ is the maximum possible).
 There is quite some work about {neighborly oriented matroids}, starting with Sturmfels~\cite{S88} and~\cite[Section 9.4]{BVSWZ99} but also more recent works such as~\cite{MP15,P13}. In the realizable setting, a $k$-neighborly matroid   $\M$ is a $k$-neighborly polytope, important mathematical objects of study in discrete geometry (i.e., a polytope such that every set of at most $k$ vertices are the vertices of a face). Also in the realizable setting, a  $k$-neighborly reorientation of  $\M$ corresponds to a permissible projective transformation that sends  $\M$ to a $k$-neighborly polytope (see \cite[Theorem 1.2]{CS85} and \cite{GL15}).



\smallskip



\subsection{The alternating oriented matroid and the value of $c_r(n,k)$.}
The \emph{cyclic polytope} of dimension $d$ with $n$ vertices 
 discovered by Carath\'eodory \cite{Car}, is the  convex hull in $\R^d$ of $n\ge d+1\ge3$ different points $x(t_1),\dots ,x(t_n)$ of the moment curve $\mu: \R\longrightarrow\R^d, \ t \mapsto (t,t^2,\dots ,t^d)$. Cyclic polytopes are neighborly polytopes and play an important role in the combinatorial  convex geometry due to their
connection with certain extremal problems, for example, the upper bound theorem due to McMullen \cite{Mac}.
The oriented matroid associated to the cyclic polytope of dimension $d=r-1$ with $n$ elements is the \emph{alternating oriented matroid} ${C}_r(n)$. It is the realizable uniform oriented matroid of rank $r$ and ground set $E=[n]:=\{1, \ldots,n\}$ such that every circuit $X\in \mathcal{C}$ is \emph{alternating}, i.e., $X_{i_j}=-X_{i_{j+1}}$ for all $1\leq j\leq r$ if $\underline{X}=\{i_1, \ldots, i_{r+1}\}$ and $i_1<\ldots<i_{r+1}$.

\smallskip

Let denote  $c_r(n,k)=f_{{C}_r(n)}(r,n,k)$, i.e., $c_r(n,k)$ is the number of $k$-neighborly reorientations of the rank r alternating oriented matroid on $n$ elements.
In 1991 Roudneff \cite{R91} proved that $c_r(n,1)\ge {2}\sum_{i=0}^{r-3}\binom{n-1}{i}$ and that $c_r(n,1)={2}\sum_{i=0}^{r-3}\binom{n-1}{i}$ if $n\ge 2r-1$.
Recently,  it was computed the exact number of $c_r(n,k)$ for $n\ge 2(r-k)+1$ and $r\ge 2k+1\ge3$\cite[Theorem 3.11]{HKM24}:  
\begin{equation}\label{formula-ciclico}
    c_r(n,k) =2\sum\limits_{i=0}^{r-1-2k} \binom{n-1}{i}.
\end{equation}



In fact, in \cite{HKM24}, Theorem 3.11 proved that $o(C_r(n),i)=2\binom{n}{r-1-2i}$ for every $i=k,\ldots,\lfloor\frac{r-1}{2}\rfloor$ and $n\ge 2(r-k)+1$, where $o(C_r(n),i)$ was defined as the
the number of $i$-neighborly reorientations of $C_r(n)$ that are not $(i+1)$-neighborly. So, for $n\ge 2(r-k)+1$  we have that $c_r(n,k)=2\sum\limits_{i=k}^{\lfloor\frac{r-1}{2}\rfloor} \binom{n}{r-1-2i}=2\sum\limits_{i=0}^{r-1-2k} \binom{n-1}{i}
$ as mentioned above. We then notice that the value of $c_r(n,k)$ is unknown for $n\le2(r-k)$.

\subsection{$k$-Roudneff's conjecture}\label{sub-roudneff}
Roudneff conjectured that the rank $r$ alternating oriented matroid on $n\ge 2r-1$ elements has the maximum number of $1$-neighborly reorientations over all rank $r$ oriented matroids on $n$ elements \cite[Conjecture 2.2]{R91}.
Here we point out that the conjecture was stated originally in terms of projective pseudohyperplane arrangements and we state here in terms of oriented matroids. Many of the combinatorial properties of projective pseudohyperplane arrangements can be studied in the language of oriented matroids. Indeed, an oriented matroid on $n$ elements of rank $r$ is naturally associated with some arrangement of $n$ (pseudo) hyperplanes in the projective space of dimension $d=r-1$, and conversely the (simple) oriented matroids 
are precisely those associated with some arrangement of pseudo hyperplanes \cite{FL78}. We refer the reader to \cite{HKM24} for more details on the translation of Roudneff's conjecture from arrangements to oriented matroids.

\begin{conjecture}[Roudneff 1991]\label{conjRoudff}
For any rank $r\ge 3$ oriented matroid $\mathcal{M}$ on $n\ge 2r-1$ elements it holds $f_{\M}(r,n,1)\le c_r(n,1)$.
\end{conjecture}

The conjecture remains open for $r\ge6$. The case $r=3$ is not difficult to prove. Ram\'irez Alfons\'in proved the case $r=4$ in 1999 \cite{R99} and recently Hern\'andez-Ortiz et al. proved it for $r=5$~\cite{HOKMS23}. In 2015 Montejano and Ram\'irez Alfons\'in proved the conjecture for Lawrence oriented matroids~\cite{MR15}.  Furthermore, B\'ar\'any et al.\cite{BBLP95} proved that for fixed $r$, the number of $1$-neighborly reorientations of any rank $r$ realizable oriented matroid on $n$ elements is $2\binom{n-1}{r-3}+2O(n^{r-4})$, i.e., Roudneff's conjecture holds asymptotically in the realizable setting (here we also point out that the authors proved this result in terms of hyperplane arrangements in the Euclidean space and not in terms of realizable oriented matroids). 
 In this manuscript, we prove  that Roudneff's conjecture holds asymptotically  for all oriented matroids (Section \ref{asymptotics},  Corollary \ref{asintotico_k-Roudf}).

\smallskip

Conjecture \ref{conjRoudff} was stated for $n\ge 2r-1$ probably because in 1991 the value of $c_r(n,1)$ was unknown for $n<2r-1$ but in 2001 was computed its value for $n\ge r+1$ \cite{FR01}. So, in view of Roudneff's conjecture, the authors asked in \cite{MR15} if Roudneff's conjectured  holds for $n\ge r+1$ \cite[Question 2]{MR15} and again, it turns out that it is true for $r\leq 5$~\cite{HOKMS23} and for  Lawrence oriented matroids~\cite{MR15}.
 For an integer $k\ge0$ it was proposed in \cite{HKM24} the following strengthening of \Cref{conjRoudff}.
\begin{conjecture}[$k$-Roudneff's conjecture]\label{conjkRoudneff} 
For any rank $r\ge 2k+1\ge 3$ oriented matroid $\mathcal{M}$ on $n\ge r+1$  elements,  $$f_{\mathcal{M}}(r,n,k)\le c_r(n,k).$$
 \end{conjecture}
\smallskip

Firstly, notice that $f_{\mathcal{M}}(r,n,k)=f_{\mathcal{M}'}(r,n,k)$ if $\mathcal{M}'\in [\M]$. So, it is enough to check Conjecture \ref{conjkRoudneff} for the reorientation classes $[\M]$ of $\M$. Secondly, every rank $r$ oriented matroid $\mathcal{M}$ on $n$ elements can be perturbed to become a uniform rank $r$ oriented matroid $\mathcal{M}'$ on $n$ elements, see~\cite[Corollary 7.7.9]{BVSWZ99} and in particular $f_{\M}(r,n,k)\le f_{\M'}(r,n,k)$ (see \cite[Theorem 3.16]{HKM24}). Thirdly, it is well-known that there is exactly one reorientation class of uniform rank $r$ oriented matroids on $n\le r+2$ elements (see \cite[Remark 1.3]{HKM24} and \cite[Section 6.1]{BVSWZ99}).  Thus, we obtain the following observation.
\begin{remark}\label{remarkk-Roudff}
To answer Conjecture \ref{conjkRoudneff}, it suffices to verify it for the reorientation classes $[\M]$ of uniform oriented matroids $\M$. Moreover, Conjecture \ref{conjkRoudneff} holds for every $n\le r+2$.
\end{remark}

For $k=1$ \Cref{conjkRoudneff} combines Roudneff's conjecture and~\cite[Question 2]{MR15}. Hence, for $k=1$ the answer is positive  if $r\leq 5$ and also holds for Lawrence oriented matroids. Conjecture \ref{conjkRoudneff} was also proved in \cite{HKM24} for $k=2$ and $r=6$, and also for every odd rank $r=2k+1$ \cite{HKM24}.

\smallskip

One might think that as in the Upper Bound Theorem~\cite{Mac} the reorientation class of $\mathcal{C}_r(n)$ is unique in attaining the maximum in \Cref{conjkRoudneff}. However, there are different reorientation classes that attain $c_r(n,k)$, for $r=5,n=8,9$ and $k=2$ and also for $r=7,n=10$ and $k=3$  \cite[Theorem 4.1]{HKM24}. Nevertheless,  for $k=1$ it was conjectured in \cite{HOKMS23}  that the reorientation class of ${C}_r(n)$ is the unique class with $c_r(n,k)$ $k$-neighborly reorientations.



\smallskip

In this work, we first obtain in Section \ref{LOM_k-Roudf}  some progress in  $k$-Roudneff's conjecture for the family of Lawrence oriented matroids (LOMs) and also provide a general upper bound of $f_{\M}(r,n,k)$ when $\M$ is a LOM. As we will see later, the class of  LOMs contains as a very particular case the alternating oriented matroid and thus a natural class to investigate the validity of \Cref{conjkRoudneff}. 
 Afterwards,  
 we prove in Section \ref{asymptotics}  that $k$-Roudneff's conjecture holds asymptotically  when $n\rightarrow \infty$.

\subsection{Organization of the paper}
The structure of the paper is as follows.
\smallskip
\begin{itemize}
    \item In Section \ref{intro} we explain some basic notions of oriented matroid theory, introduce the alternating oriented matroid and explain the $k$-Roudneff conjecture (Conjecture \ref{conjkRoudneff}).  

    \item In Section \ref{LOM_k-Roudf} we obtain some results on Conjecture \ref{conjkRoudneff} for LOMs. More precisely:

    \begin{itemize}
        \item In Subsection \ref{k-neighborly-LOMs}  we first introduce the LOMs, present some basic results and obtain an equivalent description of $k$-neighborly LOMs in terms of its acyclic reorientations ($k$-neighborly plain travels).
       
        \item  In Subsection \ref{even_r_max_k_Roudff} we prove  Conjecture \ref{conjkRoudneff} for LOMs with even rank $r=2k+2$ 
 and $n\geq 3r-4\ge 8$ (Theorem \ref{rparkmax}). 

        \item In Subsection \ref{newbounds_alternating}, we provide a general upper bound of $f_{\M}(r,n,k)$ when $\M$ is a LOM 
 (Theorem \ref{theo-cota-sup}). 

        \item In Subsection \ref{Chess-subsection} we introduce the Chessboard of a LOM in order to calculate the number of reorientation classes of LOMs (this allows us to calculate $f_{\M}(r,n,k)$ for low ranks with a computer program). 

        \item In Subsection \ref{computer r=7_k=2} we first explain a computer program that computes $f_{\M}(r,n,k)$ when $\M$ is a LOM. Using our computer program  we prove Conjecture \ref{conjkRoudneff} for LOMs for $r=7$, $k=2$ and $n=11$ (Theorem \ref{thm_computer7}), for $r=8$, $k=2,3$ and $n=11,12$ (Theorem \ref{thm_computer8}), for $r=9$, $k=2,3$ and $n=12$ and also for $r=9$, $k=3$ and $n=13$ (Theorem \ref{thm_computer9}). Then, using Proposition \ref{coro-LOMs} and Theorems \ref{thm_computer7} and \ref{thm_computer9}  we prove  Conjecture \ref{conjkRoudneff} for LOMs for the cases $r=7$, $k=2$ and $n\ge 11$ (Corollary \ref{2-roudneff}) and also for $r=9$, $k=3$ and $n\ge 12$ (Corollary \ref{3-roudneff}). 
    \end{itemize}

\item  In Section \ref{asymptotics} we prove that for fixed $r\ge 4$ and $k\ge 1$ the number of $k$-neighborly reorientations of any rank $r$ oriented matroid on $n$ elements is at most
$2\binom{n}{r-2k-1}+O(2n^{r-2k-2})$  as $n\rightarrow \infty$ (Corollary \ref{asintotico_k-Roudf}), proving $k$-Roudneff's conjecture asymptotically  and thus giving more credit to \Cref{conjkRoudneff}.
\end{itemize}

\medskip

\section{$k$-neighborly reorientations of LOMs}\label{LOM_k-Roudf}

In this section we study the parameter $f_{\M}(r,n,k)$ when $\M$  is a Lawrence oriented matroid (LOM). We obtain obtain some results on $k$-Roudneff's conjecture for LOMs.
\smallskip

Before defining a LOM, we need to introduce the chirotope of an oriented matroid. The definition of the base axioms of an oriented matroid is due to Gutierrez Novoa~\cite{Gut65} who called the structure multiply ordered sets. Lawrence~\cite{L82} proved that this yields another axiom system for oriented matroids. 
The term \emph{chirotope} of an oriented matroid 
is due to Dress~\cite{Dre85}, who reinvented oriented matroids. 
An oriented matroid of rank $r$ is a pair $\M=(E,\chi)$ of a finite ground set $E$ and a \emph{chirotope} $\chi:E^r\to \{+,0,-\}$ satisfying:
\begin{itemize}
    \item[(B0)] $\chi$ is not identically zero,
    \item[(B1)] $\chi$ is alternating,
 \item[(B2)] if $x_1,\ldots, x_r,y_1, \ldots, y_r\in E$ and $\chi(y_i,x_2,\ldots, x_r)\chi(y_1, \ldots, y_{i-1},x_1,y_{i+1},\ldots, y_r)\geq 0$ for all $i\in[r]$, then $\chi(x_1,\ldots, x_r)\chi(y_1, \ldots, y_r)\geq 0$.
\end{itemize}

If $\chi: E^r \rightarrow \{-,+\}$ is a chirotope, then $\M=(E,\chi)$ is a uniform oriented matroid.
Moreover, if $E=\{1,\ldots,n\}$ and $\chi(B)=+$ for any ordered tuple $B = (b_1<\ldots <b_r)$, then $\M=(E,\chi)$ is  ${C}_r(n)$, the alternating oriented matroid of rank $r$ on $n$ elements. It is known that every oriented matroid $\M$ has exactly
two basis orientations and these two basis orientations are opposite, $\chi$ and $-\chi$ (see \cite[Proposition 3.5.2]{BVSWZ99}).

\subsection{$k$-neighborly Lawrence oriented matroid (LOM)}\label{k-neighborly-LOMs}

A {\em Lawrence oriented matroid} (LOM) of rank $r$ on the totally ordered set $E=\{1,\ldots,n\}$, $r\le n$, is an uniform oriented matroid obtained as the union of $r$ uniform oriented matroids ${\M}_1,\ldots,{\M}_r$ of rank $1$ on $(E,<)$ (see \cite{lawrence1981unions,RS88}). They were studied and used in \cite{R01} to find examples to provide upper bounds to the McMullen problem, stated in \cite{L72}, 
and then in a generalization of this problem to $k$-neighborly polytopes \cite{GL15}. Note that yet another variant of McMullen’s problem has been studied recently in which LOMs are also used \cite{garcia2023number}. 
In \cite{MR15} they were also used to prove
Roudneff's conjecture (Conjecture \ref{conjRoudff}) for LOMs.

\smallskip

LOMs can also be defined via the signature of their bases, that is, via their chirotope $\chi$. Indeed, the chirotope $\chi$ corresponds to some LOM, ${\M}_A$, if and only if there exists a matrix $A=(a_{i,j})$, $1\le i\le r$, $1\le j\le n$ with entries from $\{+1,-1\}$ (where
the $i$-th row corresponds to the chirotope of the oriented matroid $\mathcal{M}_i$) such that
\begin{equation}\label{bases-lom}
\chi(B)=\prod_{i=1}^{r} a_{i,j_i}
\end{equation}
where $B$ is a base of ${\M}_A$, i.e.,  an ordered $r$-tuple, $j_1<\cdots<j_r$, of elements of $E$. From now on we will consider matrices $A$ with entries from $\{+1,-1\}$. 

\smallskip

First, notice that the alternating oriented matroid is a LOM,  ${C}_r(n)=\M_A$, where $A$ is the matrix with all its entries of the same sign (in fact, we point out in Subsection \ref{Chess-subsection} that $A$ is not unique).
Second, acyclic {LOMs} are realizable since they are unions of realizable oriented matroids  \cite[Proposition 8.2.7]{BVSWZ99}.
Thirdly, as well as the signs of the bases of a LOM $\M_A$ can be obtained directly from matrix $A$ (Equation (\ref{bases-lom})), it is also possible to obtain from $A$ the signs of the elements of any circuit of $\M_A$ in the following way. It is well-known that the chirotope of a uniform oriented matroid relates the basis and the circuits as follows: $\chi(B)=-X_{j_i}\cdot X_{j_{i+1}}\cdot \chi(B'),$
where $\underline{X}=\{j_1,...,j_{r+1}\}$, $B=\underline{X}\setminus j_i$ and $B'=\underline{X}\setminus j_{i+1}$
see \cite[Section 3.5]{BVSWZ99}. Thus, it can be seen that the signs of any circuit $\underline{X}=\{j_1,...,j_{r+1}\}$ in ${\M}_A$ can be obtained from matrix $A$ by the equation 
\begin{equation}\label{eq-circuits}
 X_{j_{i+1}}=-X_{j_i}a_{i,j_i}a_{i,j_{i+1}}   
\end{equation}
see \cite[Lemma 2.1]{R01}. We also note the following.
\begin{remark}\label{prop-lom}
    Let $A = (a_{i,j})$, $1\le i \leq r$, $1\le j \leq n$, be a matrix and ${\M}_A$  its corresponding Lawrence oriented matroid.
    \begin{itemize}
        \item[(i)] The coefficients $a_{i,j}$ with $i>j$ or $j-n > i-r$ do not play any role in the definition of ${\M}_A$ (since they never appear in (\ref{bases-lom})).
        \item[(ii)] The opposite chirotope $-\chi$ is obtained by reversing the sign of all the coefficients of a line of $A$.
        \item[(iii)] The oriented matroid $_{-R}{\M}_A$ is obtained by reversing the sign of all the coefficients of a set of columns $R$ in $A$. 
    \end{itemize}
\end{remark}

The Top and Bottom  travels of a matrix $A$ were defined and studied in \cite{R01} to detect when $\M_A$ is $1$-neighborly. Next, we will use them to determine
when $\M_A$ is $k$-neighborly (\Cref{LOM_k_neigh}). 

\smallskip

The \emph{Top Travel} $TT$ of $A$ is a subset of the entries of $A$,
$$\{ [a_{1, 1}, a_{1, 2},\dots ,a_{1, j_{1}}] , [a_{2, j_{1}}, a_{2, j_{1}+1},\dots , a_{2, j_{2}}] , \dots , [a_{s, j_{s-1}}, a_{s, j_{s-1}+1},\dots , a_{s, j_{s}}] \},$$
with the following characteristics.

\begin{enumerate}
\item $a_{i, j_{i-1}} \times a_{i, j}= 1, \text{ for all } \  j_{i-1} \leq j < j_{i}$, $1\le i\le s$, where we define $j_0=1$;

\item ${a_{i, j_{i}-1}} \times a_{i, j_{i}}= -1,   \text{ for all } \  1\le i\le s-1$  and either

\begin{enumerate}
\item $j_{s-1}=n$; then $[a_{s, j_{s-1}}, a_{s, j_{s-1}+1},\dots , a_{s, j_{s}}]=[a_{s, j_{s}}]$ or

\item $j_{s-1}<n$  and $1\leq s < r$;   then  $j_{s} = n$  and $a_{s, j_{s}-1} \times a_{s, j_s}= 1$ or

\item  $j_{s-1}<n$, $s=r$ and $ j_{s} < n$; then $a_{s, j_{s}-1} \times a_{s, j_s}= -1$ or

\item  $j_{s-1}<n$, $s=r$ and $ j_{s} =n$; then $a_{s, j_{s}-1} \times a_{s, j_s}=1$ or $-1$.
\end{enumerate}
\end{enumerate}

\medskip

The \emph{Bottom Travel} $BT$ of $A$ is a subset of the entries of $A$,
$$\{ [a_{r, n}, a_{r, n-1},\dots ,a_{r, j_{r}}] , [a_{r-1, j_{r}}, a_{r-1, j_{r}-1},\dots , a_{r-1, j_{r-1}}] , \dots , [a_{s, j_{s+1}}, a_{s, j_{s+1}-1},\dots , a_{s, j_{s}}] \},$$
 with the following characteristics.

\begin{enumerate}
\item $ a_{i, j_{i+1}} \times a_{i, j}= 1, \text{ for all }\  j_{i} < j \leq j_{i+1}$, $s\le i\le r$, where we define $j_{r+1}=n$;
\item ${a_{i, j_{i}+1}} \times a_{i, j_{i}}= -1, \text{for all } \  s-1\le i\le r $ and either
\begin{enumerate}
\item $j_{s+1}=1$;   then  $[a_{s, j_{s+1}}, a_{s, j_{s+1}-1},\dots , a_{s, j_{s}}]=[a_{s, j_{s}}]$ or
\item $j_{s+1}>1$ and $1 < s \leq r$; then  $j_{s} = 1$ and  $a_{s, j_{s}+1} \times a_{s, j_s}= 1$ or
\item $j_{s+1}>1$, $s=1$ and $1<j_{s}$;  then  $a_{s, j_{s}+1} \times a_{s, j_s}=-1$ or
\item $j_{s+1}>1$, $s=1$ and $1=j_{s};$ then $a_{s, j_{s}+1} \times a_{s, j_s}=1$ or $-1$.
\end{enumerate}
\end{enumerate}
\medskip

In other words, $TT$ (resp. $BT$) may be thought of as a travel starting at $a_{1,1}$ (resp. at $a_{r,n}$) making horizontal movements to the right (resp. to the left) and vertical movements to the bottom (resp. to the top) of $A$ according with the above constructions.
Notice that any matrix $A$ admits exactly one $TT$ and one $BT$. Figures \ref{ejemplo} and \ref{ejemplo2} show examples of matrices with its Top and Bottom travels.

\smallskip

We say that the $TT$ [resp. $BT$] of $A$ is \emph{positive} if $a_{r, i-1},a_{r, i}\in TT$ [$a_{1, i},a_{1, i+1}\in BT$] and  $a_{r, i-1} \times a_{r, i}=-1$ [$a_{1, i} \times a_{1, i+1}=-1$] for some $i\le n$ [$i\ge 1$] (see matrix $A_S$ of Figure \ref{ejemplo2}). Notice that $TT$ is positive if and only if $BT$ is positive. Moreover, using Equation (\ref{eq-circuits}) we may deduce the following.


\begin{remark}\label{rem-positive}
    A LOM $\mathcal{M}_A$ is not acyclic (i.e., has a positive circuit) if and only if $A$ has a positive top travel.
\end{remark}

Given a set $S$ of columns, we denote by $A_S=(a_{i,j}^S)$
the matrix obtained from reorienting the column of $S$ and denote by $TT_S$ (resp. $BT_S$) the top travel (bottom travel) of $A_S$.
Figure \ref{ejemplo2} shows a matrix $A_S$ with its positive Top and Bottom travels. Notice that the oriented matroid $\M_{A_S}$ is not acyclic since, for example, the circuit $(++++++0)$ is positive.  
\smallskip

We say that the matrix $A$  is \emph{$k$-neighborly} if for any set $S$ of at most $k$ columns,  $TT_S$ is not positive. So, we have the following result.

\begin{lemma}\label{LOM_k_neigh}
  A matrix $A$  is $k$-neighborly if and only if $\M_A$ is $k$-neighborly.
\end{lemma}
\begin{proof}
Reorientations of columns in $A$ are just reorientations of elements in $\M_A$ (Remark \ref{prop-lom} (iii)).  Let $S$ be any set of at most $k$ columns, then by Remark \ref{rem-positive} $TT_S$ is not positive in $A_S$ if and only if  $_{-S}\M_A$ is acyclic. Finally, by definition $_{-S}\M_A$ is acyclic if and only if $\M_A$ is $k$-neighborly.
\end{proof}

Figure \ref{ejemplo2} shows a $1$-neighborly matrix $A$. Then $\M_A$ is a $1$-neighborly oriented matroid by the above observation. Nevertheless,  $A$ (and thus $\M_A$) is  not $2$-neighborly since $\M_{A_S}$ is acyclic for $S=\{2,3\}$.

\smallskip

The $k$-neighborly reorientations  of $\M_{A}$ can be identified with yet another simple object.
A \emph{Plain Travel} $P$ of  $A$ is a subset of the entries of $A$ of the form
$$P=\{ [a_{1, 1}, a_{1, 2},\dots ,a_{1, j_{1}}] , [a_{2, j_{1}}, a_{2, j_{1}+1},\dots , a_{2, j_{2}}] , \dots , [a_{s, j_{s-1}}, a_{s, j_{s-1}+1},\dots , a_{s, j_{s}}] \} $$  with $ 2\leq j_{i-1}  {<} j_{i}  \leq n$ for all $1 \leq i \leq r ,\;\; 1< s \leq r$   and  $j_{s} = n$.
In other words, $P$ starts with $a_{1,1},a_{1,2}$, it cannot make two consecutive vertical movements and ends at column $n$ in some row of $A$ (see Figures \ref{ejemplo} and \ref{ejemplo2}).

\smallskip


Given a plain travel $P$ of $A$, let denote by $\mathcal{A}$ the new matrix obtaining by reorienting a set of column of $A$  such that its top travel follows the same path as $P$. If there is no confusion, we will denote the top travel of $\mathcal{A}$ simply as $P$. 
It turns out that the set of all plain travels of $A$ are in one-to-one correspondence with the set of acyclic reorientations (i.e., $0$-neighborly reorientations) of $\M_{A}$.

\begin{lemma}\cite[Lemma 3.1]{R01}\label{PT_and_reor}
There is a bijection between the set of all plain travels of the matrix $A$ and the set of all acyclic reorientations  of $\M_{A}$.
\end{lemma}
The bijection stated above is defined by associating to each plain travel $P$ the set of columns of $A$ that should be reoriented in order to transform $A$ into $\mathcal A$  (see Figure \ref{ejemplo}). 

\begin{figure}[htb]
\begin{center}
 \includegraphics[width=1\textwidth]{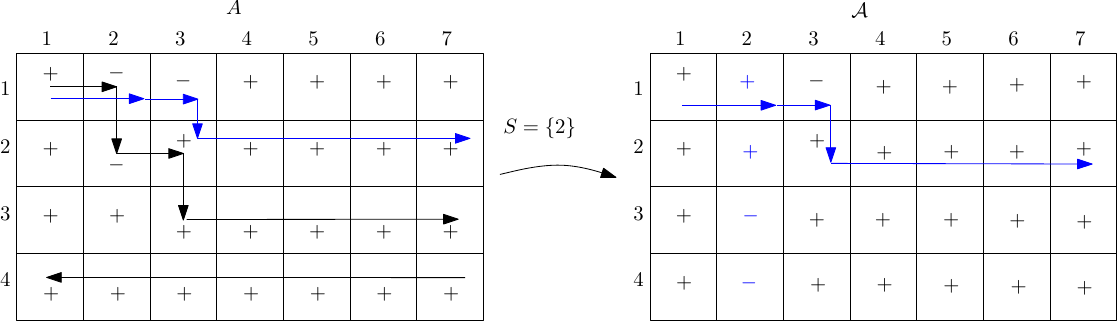}
\caption{In black, the top and bottom travels of the matrix $A$ and in blue, a plain travel $P$ of $A$. The matrix $\mathcal{A}$ is obtaining by reorient column $2$ of $A$. Notice that $P$ is the top travel of $\mathcal{A}$.} \label{ejemplo}
\end{center}
\end{figure}
 We say that a plain travel $P$ of a matrix $A$ is \emph{$k$-neighborly}  if  $\mathcal A$ is a $k$-neighborly matrix. The following lemma  will be useful in this section. 
\begin{lemma}\label{PT_and_reor_k}
    A plain travel $P$ of matrix $A$  is $k$-neighborly if and only if its corresponding acyclic reorientation of  $\M_{A}$ is $k$-neighborly.
\end{lemma}
\begin{proof}
 $P$  is a $k$-neighborly plain travel of $A$ if and only if $\mathcal{A}$ is a $k$-neighborly matrix if and only  if  $\mathcal{M}_{\mathcal{A}}$ is a $k$-neighborly oriented matroid by Lemma \ref{LOM_k_neigh}. Since plain travels of $A$ corresponds to  acyclic reorientations of  $\M_{A}$ (Lemma \ref{PT_and_reor}), the plain travel $P$ corresponds to an acyclic reorientation $R$ of $\M_{A}$. Finally, notice that $_{-R}\M_{A}=\mathcal{M}_{\mathcal{A}}$ which is $k$-neighborly. Therefore,  $R$ is a $k$-neighborly acyclic reorientation of $\M_{A}$, concluding the proof.
\end{proof}

Figure \ref{ejemplo2} shows in blue a plain travel $P$ of $A$. Observe that $\mathcal{A}$ (the matrix obtaining by reorienting a set of columns of $A$ whose top travel is $P$) is not the matrix $A_S$ of Figure \ref{ejemplo2} (in fact $\mathcal{A}$ can be obtained by reorienting column 2). Now, notice that for each set $S'$ of at most $2$ columns of $\mathcal{A}$, the top travel of $\mathcal{A}_{S'}$ is not positive. Therefore, the plain travel $P$ of $A$ is $2$-neighborly, which, in turn, corresponds to a $2$-neighborly reorientation of $\M_{A}$ by Lemma \ref{PT_and_reor}.

\begin{figure}[htb]
\begin{center}
 \includegraphics[width=1\textwidth]{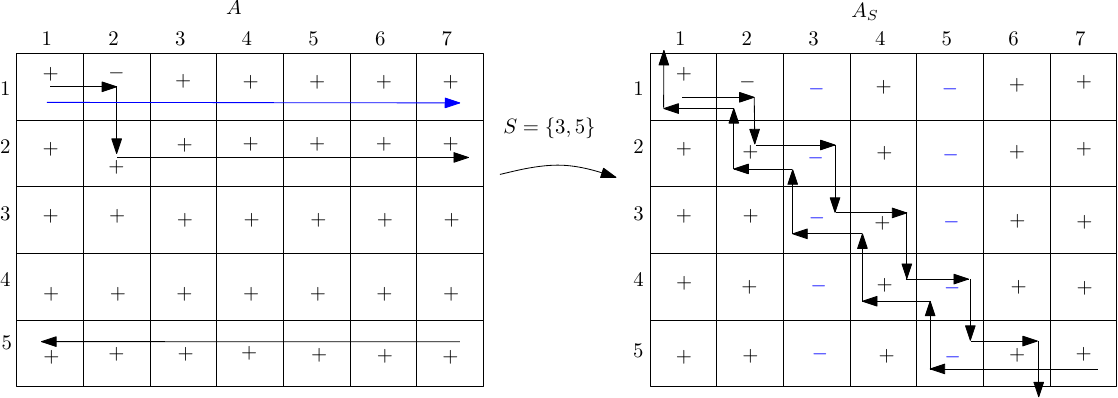}
\caption{In black, the top and bottom travels of the matrices $A$ and $A_S$ and in blue, a plain travel $P$ of $A$. The matrix $A_S$ is obtaining by reorient columns $3$ and $5$ of $A$. The $TT$ and the $BT$ of $A_S$ are positive.} \label{ejemplo2}
\end{center}
\end{figure}

As mentioned at the begining of this section, any oriented matroid has two basis orientations (which are opposite). Thus, we have the following observation which will be used later.

\begin{remark}\label{remark-2}
  Each $k$-neighborly plain travel of a matrix $A$ corresponds to exactly two $k$-neighborly reorientations of $\M_{A}$.
\end{remark}
In fact, these two reorientations are $_R\M_{A}$ and $_{E\setminus R}\M_{A}$ for some $R\subset E$, where $E$ is the ground set of $\M_A$.

\subsection{$k$-Roudneff's conjecture for LOMs when $r=2k+2$}\label{even_r_max_k_Roudff}  

In this subsection we answer Conjecture \ref{conjkRoudneff} in the affirmative for LOMs when $r=2k+2\ge 4$ and $n\geq 3r-4$ (Theorem \ref{rparkmax}). The case $r=2k+1\ge 3$ and $n\ge r+2$ was proved for any oriented matroid in \cite[Corollay 3.17]{HKM24}.



\smallskip

Let denote by $PT$ the set of plain travels of a matrix $A$ and let $PT^k$ be the set of plain travels of $A$ that are $k$-neighborly. More generally, given any set $\mathcal{P}$ of plain travels of $\mathcal{M}$, we denote by $\mathcal{P}^k$ the subset of plain travels of $\mathcal{P}$ that are $k$-neighborly. We notice that for a LOM $\mathcal{M}$, we have that $f_\mathcal{M} (r,n,k)=2|PT^k|$ by Lemma \ref{PT_and_reor_k} and Remark \ref{remark-2}. Thus, we may rewrite \cite[Corollay 3.17]{HKM24} now in terms of $k$-neighborly plain travels as follows.

\begin{corollary}\label{r_odd_max_ort-lom}
    Let $A$ be a $(r\times n)$-matrix with $r=2k+1\ge 3$ and $n\ge r+2$, then $|PT^k|\le 1$.
\end{corollary}

The next lemmas will be useful to prove Theorem \ref{rparkmax}.

\begin{lemma}\label{n_grande_baja_de_2_en_2}
    Let $A$ be a $(r\times n)$-matrix with $r=2k$ and $n\ge 3k+1\ge 4$ and let $TT$ be its top travel. Then there exists $S\subset\{2,\ldots,n\}$ with $|S|\le k$  such that $TT_S$ is positive in $A_S$.
\end{lemma}

\begin{proof}
We will prove the lemma by induction on $k$ where the case $k=1$ can be easily checked (see Figure \ref{ejemplo3}). Thus, assume that the Lemma holds for every $k'\ge 1$ and let $k=k'+1$. By induction hypothesis on $k'$, there exists  $S'\subset \{2,\ldots, 3k'+1\}=\{2,\ldots, 3k-2\}$ with $|S'|\le k'=k-1$ such that $TT_{S'}$ arrives in $A_{S'}$ at row $2k'+1=2k-1$ and column $j$ with $j\leq 3k-2$.

\smallskip

Recall that $a^{S'}_{i,j}$ are now the entries of matrix $A_{S'}$. First suppose that $a^{S'}_{2k-1,3k+1}\in TT_{S'}$, then $a^{S'}_{2k-1,j}\in TT_{S'}$ for every $j\in \{3k-2, 3k-1, 3k, 3k+1\}$ and so $a^{S'}_{2k-1,3k-2} = a^{S'}_{2k-1,3k-1}=a^{S'}_{2k-1,3k}=a^{S'}_{2k-1,3k+1}$. By the pigeonhole principle there exists $l\in\{3k-1,3k\}$ such that $l$ is the minimum column index such that $a^{S'}_{2k,j} = a^{S'}_{2k,j'}$ for some $l'\in\{3k,3k+1\}$. Then, consider $S = S'\cup \{l\}$ and observe that $TT_{S}$ is positive in $A_S$ and $1\notin S$. Now suppose that $a^{S'}_{2k-1,3k+1}\notin TT_{S'}$. If $TT_{S'}$ is not positive in $A_{S'}$ then $a^{S'}_{2k,3k}, a^{S'}_{2k,3k+1}\in TT_{S'}$ and $a^{S'}_{2k,3k}=a^{S'}_{2k,3k+1}$. Thus, consider $S=S'\cup\{3k+1\}$ and notice that $TT_{S}$ is positive  
in $A_S$ where $1\notin S$, concluding the proof.
\end{proof}

\begin{figure}[htb]
\begin{center}
 \includegraphics[width=0.6\textwidth]{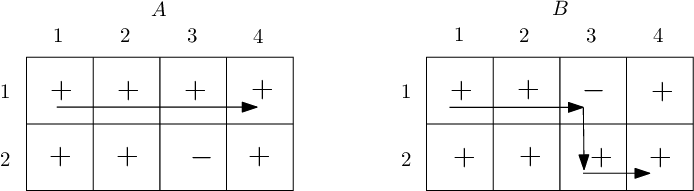}
\caption{If $a_{1,4}\in TT$, then depending on the signs of $a_{2,2}$, $a_{2,3}$ and $a_{2,4}$, we must reorient column $2$ or $3$. In the first example (matrix $A$), we must reorient column $2$ in order to make its top travel positive. If $a_{1,4}\notin TT$ (see matrix B), then we must reorient at most one column, for example column 4 (if necessary), in order to make the $TT$ positive. } \label{ejemplo3}
\end{center}
\end{figure}

Given a $(r\times n)$-matrix $A$, we denote by $PT_j$, with $2\le j \le n+1$, the set of plain travels of $A$ such that their first vertical movement is at column $j$ and where $P_{n+1}$ is the plain travel that does not make vertical movements. We define $PT_{\geq j} = \bigcup\limits_{i=j}^{n+1}PT_i$ and $PT_{\ge j}^{k}=\bigcup\limits_{i=j}^{n+1}PT^k_i$.

\smallskip

Given a plain travel $P$ of $A$, recall that $\mathcal{A}$ is the matrix obtained by reorienting a set of columns of $A$ whose top travel is $P$. Moreover, if $S$ is a set of columns, then the matrix $\mathcal{A}_S$ is the matrix obtained from $\mathcal{A}$ reorienting the columns of $S$ with top travel $P_S$.

\begin{lemma}\label{PTjk}
Let $A$ be a $(r\times n)$-matrix with $r=2k+2$ and $n\ge 3k+3$. Then $|PT_{j}^{k}|\leq 1$ for every $2\leq j\leq n-(3k+1)$.
\end{lemma}

\begin{proof}
Let $V_j$  (resp. $H_j$) be the set of plain travels of $ PT_j$ that make a vertical movement at column $j+1$ from rows $2$ to $3$ (resp. do not make a vertical movement at column $j+1$ from rows $2$ to $3$), see Figure \ref{ejemplo-vj}. Since $V_j$ and $H_j$ is a partition of $PT_j$, in particular we have that $|PT_j^k|=|V_j^k|+|H_j^k|$. 

\begin{figure}[htb]
\begin{center}
 \includegraphics[width=0.4\textwidth]{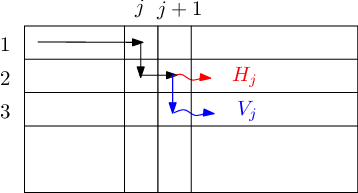}
\caption{The sets $V_j$ and $H_j$. } 
\label{ejemplo-vj}
\end{center}
\end{figure}

We first claim that $|H_{j}^{k}|\leq 1$. Consider the submatrix $B$ of $A$ obtained by deleting row 1 and the set of columns $\{1,\ldots, j\}$  and observe that the plain travels of $H_j$ restricted to $B$ coincides with all the plain travels of $B$. Since $B$ is a $(r'\times n')$-matrix with $r'=2k+1$ and $n'\geq r'+2$, then there is at most one $k$-neighborly plain travel in $B$ by Corollary \ref{r_odd_max_ort-lom}, concluding also the same for $A$, i.e.,
$|H_{j}^{k}|\leq 1$.

We now claim that $|V_j^k|=0$. Consider any $P\in V_j$ and let $\mathcal{A}$ be the  matrix obtaining by reorienting a set of columns of $A$ whose top travel is $P$. Now, consider the submatrix $\mathcal{B}$ of $\mathcal{A}$ obtained by deleting the set of rows $\{1,2\}$ and the set of columns $\{1,\ldots,j\}$ and observe that $P$ is also the top travel of $\mathcal{B}$. Applying Lemma \ref{n_grande_baja_de_2_en_2} to $\mathcal{B}$, there exists a set of columns $S\subset \{j+2,\ldots,n\}$ with $|S|\leq k$ such that $P_S$ is positive in $\mathcal{B}_S$. Notice that $P_S$ is also the top travel of $\mathcal{B}_S$ since $S\subset \{j+2,\ldots,n\}$. Then, $P_S$ is  positive in $\mathcal{A}_S$ obtaining that $P$ is not $k$-neighborly in $\mathcal{A}$. Hence, $|V_j^k|=0$.

Therefore $|PT_j^k|=|V_j^k|+|H_j^k|\leq 1$, concluding the proof.
\end{proof}

\begin{lemma}\label{Extesionk}
Let $A$ be a $(r\times n)$-matrix with $r=2k+2$, $n\geq r+1$ and
suppose that for some $j\geq 2$ and every $P\in PT_{\geq j}$ there exists $S\subseteq \{1,\ldots,j-1\}$ with $|S|\leq k$ such that $P_S$ arrives in $\mathcal{A}_S$ at row $r$ in at most column  $j-1$. Then $|PT_{\geq j}^{k}|\leq 1$.
\end{lemma}
\begin{proof}
First notice that the set $S$ is independent of the choice of $P\in PT_{\geq j}$ since all of them do not make vertical movements in the first $j-1$ columns (see Figure \ref{ejemplo-TEO2}). Suppose by contradiction that $|PT_{\geq j}^{k}|\geq 2$ and consider $P,P'\in PT_{\geq j}^{k}$ different. Let $\mathcal{A}_S=(a_{i,j})$ and $\mathcal{A}_S'=(b_{i,j})$ be the matrices in which $P_S$ and $P'_S$ are the top travel,  respectively.
Since $P$ and $P'$ coincide in the first $j-1$ columns, $P_S$ and $P'_S$ will coincide as well in the first $j-1$ columns, since  $S\subseteq \{1,\ldots,j-1\}$. Nevertheless, as $P_S$ and $P'_S$ are different, it follows that there exists $l\in\{j,\ldots,n\}$ such that $a_{r,l}\neq b_{r,l}$ since by hypothesis both arrive at row $r$. Hence, either $P_S$ or $P'_S$ is positive, concluding that either $P$ or $P'$ is not $k$-neighborly and so, arriving  a contradiction. Therefore $|PT_{\geq j}^{k}|\leq 1$.
\end{proof}

\begin{figure}[htb]
\begin{center}
 \includegraphics[width=0.5\textwidth]{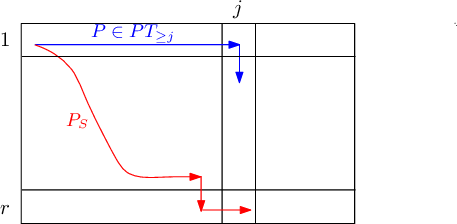}
\caption{The situation of Lemma \ref{Extesionk}} 
\label{ejemplo-TEO2}
\end{center}
\end{figure}



Now we are ready to prove $k$-Roudneff's conjecture for LOMs with even rank  $r=2k+2$.

\begin{theorem}\label{rparkmax}
Let $\M$ be a rank $r=2k+2\ge 4$ LOM on $n\geq 6k+2=3r-4$ elements, then $$f_{\M}(r,n,k)\leq c_{r}(n,k).$$
\end{theorem}
\begin{proof}
We will prove that $f_{\M}(r,n,k)\leq  2n$ since $c_r(n,k)=2n$ by Equation (\ref{formula-ciclico}). As $\M$ is a LOM, then $\M=\M_A$ for some $(r\times n)$-matrix $A$. 
By Lemma \ref{PT_and_reor_k}, $k$-neighborly plain travels of $A$ correspond to $k$-neighborly reorientations of $\M_A$. Then, it is enough to prove that $|PT^k_{\geq 2}|\le n$ by Remark \ref{remark-2}.
\vspace{0.2cm}

If $j\in \{2,\cdots,3k+1\}$, then $|PT_{j}^{k}|\leq 1$ by Lemma \ref{PTjk}, concluding that $|\bigcup\limits_{j=2}^{3k+1} PT_j^k|\le 3k$. Now suppose that $j\in\{3k+2,\cdots,n+1\}$. 
For each $P\in PT_{\ge 3k+2}$ let $\mathcal{A}$ be the matrix obtained by reorienting a set of columns of $A$  whose top travel is $P$ and consider the submatrix $\mathcal{A}'$ of $\mathcal{A}$ obtained by deleting rows $r-1$ and $r$ and the columns
$3k+2,\cdots,n$ (see Figure \ref{ejemplo-TEO3}). Note that regardless of the choice of $P\in PT_{\ge 3k+2}$, the matrix $\mathcal{A}$ (and so $\mathcal{A}'$) is the same in the first $3k+1$ columns since each $P$ coincides on that part of $\mathcal{A}$. 
As a consequence of Lemma \ref{n_grande_baja_de_2_en_2} applied to $\mathcal{A}'$, it follows that there exists a set of columns $S\subseteq \{2,\ldots, 3k+1\}$, with  $|S|\le k$, such that $P_S$ is positive in $\mathcal{A}'_S$. Hence,  $P_S$ arrives in $\mathcal{A}_S$ at row $r-1$ in at most the column $3k+1$ for every $P\in PT_{\ge 3k+2}$. If $P_S$ also arrives in $\mathcal{A}_S$ at row $r$ in at most the column $3k+1$, then by Lemma \ref{Extesionk} we obtain that $|PT^k_{\ge 3k+2}|\le 1$ and so $|PT^k_{\ge 2}|\le 3k+1< n$, concluding the proof. Otherwise, let $\mathcal{S}$ be the set of non positive top travels $P_S$ of $\mathcal{A}_S$ such that $P\in PT_{\ge 3k+2}$ and observe that $P_S$ makes at most one vertical movement in $\mathcal{A}_S$ from rows $r-1$ to $r$ at some of the columns $3k+2,\cdots,n$, obtaining that $|\mathcal{S}|\le n-(3k+1)+1=n-3k$. Finally, since $|PT^k_{\ge 3k+2}|\le |\mathcal{S}|$ we conclude that $|PT^k_{\ge 2}|\le n$ and the theorem holds.
\end{proof}
\begin{figure}[htb]
\begin{center}
 \includegraphics[width=0.6\textwidth]{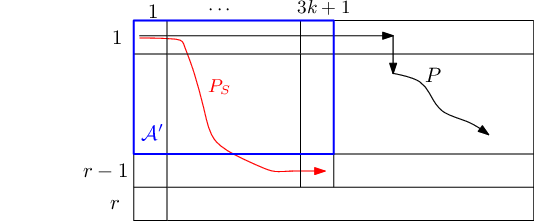}
\caption{The case $j\in\{3k+2,\cdots,n+1\}$ in Theorem \ref{rparkmax} } \label{ejemplo-TEO3}
\end{center}
\end{figure}

Finally, we notice that  Conjecture \ref{conjkRoudneff} has not been solved for LOMs with  $r=2k+2$ and $n\in \{r+3,\ldots,3r-5\}$.

\medskip
\subsection{A general upper bound for $f_\mathcal{M}(r,n,k)$}\label{newbounds_alternating}

Next, we provide a general upper bound of $f_\mathcal{M}(r,n,k)$ when $\mathcal{M}$ is a LOM (Theorem \ref{theo-cota-sup}). First, we need to prove the following lemma.

\begin{lemma}\label{escalera}
Let $A$ be a $(r\times n)$-matrix with $n\ge 3k+1$, $r\ge 2k+1$ and let $P$ be a plain travel of $A$. Then the following conditions hold:
\begin{itemize}

\item[(1)] If $P$ makes a vertical movement from row $r-2k$ to $r-2k+1$ at column $j\leq n-3k$, then $P$ is not $k$-neighborly.

\item[(2)] If $P$ makes a vertical movement from row $r-2k+1$ to $r-2k+2$ at column $j\leq n-3k+2$, then $P$ is not $k$-neighborly.
\end{itemize}
\end{lemma}

\begin{proof}

(1)  Consider the submatrix $\mathcal{A}'$ of $\mathcal{A}$ obtained by deleting the set of rows $\{1,\ldots,r-2k\}$ and the set of columns $\{1,\ldots,j-1\}$ from $\mathcal{A}$ and notice that $P$ restricted to $\mathcal{A}'$ is the top travel of $\mathcal{A}'$.
 As $\mathcal{A}'$ is a $(r'\times n')$-matrix with $r'=2k$ and $n'\geq 3k+1$, then applying Lemma \ref{n_grande_baja_de_2_en_2} to matrix $\mathcal{A}'$, there exists $S\subseteq \{j+1,\ldots,n\}$ with $|S|\leq k$ such that the restriction of $P_S$ in $\mathcal{A}'_S$ is positive in $\mathcal{A}'_S$. As $j\notin S$ we notice that $P_S$ is positive in $\mathcal{A}_S$,
 concluding that $P$ is not $k$-neighborly.
\vspace{0.2cm}

(2) Consider the submatrix $\mathcal{A}'$ of $\mathcal{A}$ obtained by deleting the set of rows $\{1,\ldots,r-2k+1\}\cup\{r\}$ and the set of columns $\{1,\ldots,j-1\}\cup\{n\}$ from $\mathcal{A}$. Notice that $P$ restricted to $\mathcal{A}'$ is the top travel of $\mathcal{A}'$. Let $k'=k-1$ and observe $\mathcal{A}'$ is a  $(r'\times n')$-matrix with $r'=2k'$ and $n'\geq 3k'+1$. Hence, applying Lemma \ref{n_grande_baja_de_2_en_2} to matrix $\mathcal{A}'$, there exists $S\subseteq \{j+1,\ldots,n-1\}$ with $|S|\leq k'$ such that the restriction of $P_{S}$ in $\mathcal{A}'_S$ is positive in $\mathcal{A}'_S$ and so, $P_{S}$ arrives at row $r$ at some column $i<n$ in $\mathcal{A}_S$ (see Figure \ref{ejemplo-TEO4} (a)). 
The above is true since $j \not\in S$, otherwise we can not assure that $P_S$ concides in $\mathcal{A}_S$ and $\mathcal{A}'_S$ (see Figure  \ref{ejemplo-TEO4} (b)). 
If $P_{S}$ is positive in $\mathcal{A}_{S}$, then $P$ is not $(k-1)$-neighborly and the result follows. Otherwise, consider $S'=S\cup\{n\}$ and notice that $P_{S'}$ is positive in  $\mathcal{A}_{S'}$ since $n\not \in S$. Therefore, $P$ is not $k$-neighborly.
\end{proof}
Figure \ref{ejemplo-TEO4} shows in (a) the submatrix
$\mathcal{A}'$  described in Lemma \ref{escalera} (2) where $P_{S}$ arrives at row $r$ at some column $i<n$. On the other hand, we notice in (b) that if column $j=5\in S$, then we can not assure that $P_{S}$ lies in $\mathcal{A}'$.
\begin{figure}[htb]
\begin{center}
 \includegraphics[width=1\textwidth]{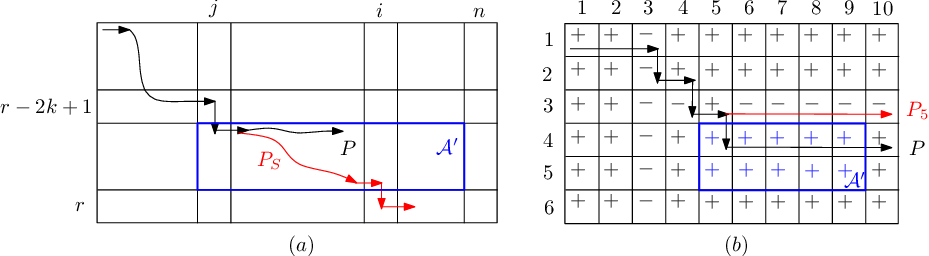}
\caption{The situation of Lemma \ref{escalera} (2)}
 \label{ejemplo-TEO4}
\end{center}
\end{figure}

The following observation will be used to prove Theorem \ref{theo-cota-sup}.

\begin{remark}\label{num_total_PT}
  The total number of plain travels of a $(r\times n)$-matrix  is $\sum\limits_{i=0}^{r-1} {n-1 \choose i}$.
\end{remark} 
\begin{proof} Any plain travel $P$ can make at most $r-1$ vertical movements and there are $n-1$ possibilities to choose $P$ to make  a vertical movement. 
 \end{proof}

\begin{theorem}\label{theo-cota-sup}
Let $\mathcal{M}$ be a rank $r\ge 2k+1\ge 3$ LOM on $n\geq 2r-1$ elements. Then 
$$f_\mathcal{M}(r,n,k) \leq c_r(n,k)+ 2\Biggl(\sum_{i=r-2k}^{r-1} \dbinom{n-1}{i} - \sum_{i=4}^{2k+3} \dbinom{n-3\lfloor\dfrac{i-1}{2}\rfloor+1+(i\; \textrm{mod} \; 2)}{r+3-i}\Biggl) $$
\end{theorem}

\begin{proof}

Let $A$ be the $(r\times n)$-matrix such that $\M=\M_A$, by Lemma \ref{PT_and_reor_k} and Remark \ref{remark-2} it is enough to prove that
$$|PT^k|\leq \dfrac{c_r(n,k)}{2} +\sum_{i=r-2k}^{r-1} \dbinom{n-1}{i} - \sum_{i=4}^{2k+3} \dbinom{n-3\lfloor\dfrac{i-1}{2}\rfloor+1+(i\; \textrm{mod} \; 2)}{r+3-i} $$

For every $i\in \{1,\cdots,k\}$, let denote by $\mathcal{P}_i$ (resp. $\mathcal{P}'_i$ ) the set of plain travels of $A$ that make exactly $r-2i+1$ vertical movements (exactly $r-2i$ vertical movements) in the set of columns $\{2,\ldots,n-3i+2\}$ (resp. in the set of columns $\{2,\ldots,n-3i\}$).
\vspace{0,2cm}
Firstly, notice that $|\mathcal{P}_i|=\dbinom{n-3i+1}{r-2i+1}$ and $|\mathcal{P}'_i|=\dbinom{n-3i-1}{r-2i}$. Second, since $n\ge 2r-1>3k+1$ and $r\ge 2k+1$, by Lemma \ref{escalera}  each  $P\in \mathcal{P}_i\cup \mathcal{P}'_i$ is not $i$-neighborly.
Moreover, notice that 
$\mathcal{P}_i\cap \mathcal{P}_j=\mathcal{P}_i\cap \mathcal{P}'_j=\mathcal{P}'_i\cap \mathcal{P}'_j=\emptyset$ for every $i,j\in \{1,\cdots,k\}$. Thirdly, if a plain travel is not $i$-neighborly then it will not be  $j$-neighborly either, if $i<j$. \\

Therefore, we conclude that $A$ has at least
$\sum\limits_{i=1}^{k} \dbinom{n-3i+1}{r-2i+1}+\sum\limits_{i=1}^{k} \dbinom{n-3i-1}{r-2i}$ plain travels that are not $k$-neighborly. Furthermore, it can be deduced that 
$$\sum\limits_{i=1}^{k} \dbinom{n-3i+1}{r-2i+1}+\sum\limits_{i=1}^{k} \dbinom{n-3i-1}{r-2i}= \sum_{i=4}^{2k+3} \dbinom{n-3\lfloor\dfrac{i-1}{2}\rfloor+1+(i\; \textrm{mod} \; 2)}{r+3-i}$$

Since the total number of plain travels of $A$ is $\sum\limits_{i=0}^{r-1} \dbinom{n-1}{i}$ (Remark \ref{num_total_PT}), we conclude that
$$|PT^k|\le \sum\limits_{i=0}^{r-1} \dbinom{n-1}{i} - \sum\limits_{i=4}^{2k+3} \dbinom{n-3\lfloor\dfrac{i-1}{2}\rfloor+1+(i\; \textrm{mod} \; 2)}{r+3-i}$$
Finally, as $n\ge 2r-1\ge 2(r-k)+1$ and $r\ge 2k+1\ge3$ then  $\dfrac{c_r(n,k)}{2} =\sum\limits_{i=0}^{r-1-2k} \binom{n-1}{i}$ (Equation (\ref{formula-ciclico})), concluding the proof.
\end{proof}
\medskip

To give an example, applying the above theorem for $r=8$, $k=2,3$ and $n=15$, we obtain the upper bounds (best known to date) of $f_\mathcal{M}(r,n,k)$ when $\M$ is a LOM, which are $f_{\M}(8,15,2)\le  c_8(15,2) + 13876$ and $f_{\M}(8,15,3)\le  c_8(15,3) + 14696$.


\subsection{Number of reorientation classes of LOMs }\label{Chess-subsection}

In this subsection we calculate the number of reorientation classes of rank $r$ LOMs on $n$ elements. This will allow us in the next subsection to solve k-Roudneff's conjecture for low ranks with the help of the computer.
\medskip

Notice that for a LOM ${\M}$ there could be different matrices $A$ such that $\M={\M}_{A}$.  For instance,  the alternating oriented matroid can be obtaining by considering the matrix that has all its entries positive or also considering the matrix that has the same sign in each row. This is due to Equation (\ref{bases-lom}) and more generally we observe the following

\begin{remark}
    If matrix $A$ is obtaining from matrix $A'$ by reversing the sign of all the coefficients of a set of rows of $A'$ then $\M_{A}={\M}_{A'}$.
\end{remark}

Therefore, reorienting a set of rows of $A$ does not change the oriented matroid $\M_{A}$. On the other hand, reorienting a set of columns of $A$ does not change the reorientation class $[{\M}_{A}]$. Thus, we have the following observation.

\begin{remark}\label{col_y_filas}
 ${\M}_{A'}\in [{\M}_{A}]$ if and only if there exists a set of columns and (perhaps) rows of $A'$ such that reorienting them in $A'$ yields the matrix $A$. 
\end{remark} 

Figure  \ref{chess} shows an example of two matrices $A$ and $A'$ with $\M_A'\in [\M_A]$ where it is necessary to reorient a column and a row of $A$ to obtain $A'$ (column $7$ and row $4$).
\vspace{0.2cm}

In order to compute the number of reorientation classes of rank $r$ LOMs on $n$ elements we need to introduce the chessboards. 
The \emph{chessboard} $B[A]$ of a  $(r\times n)$-matrix $A$ can be constructed from the entries of $A$.  It is defined by a black and white board of size $(r-1) \times (n-1)$, such that the square $s(i,j)$ has its upper left hand corner at the intersection of row $i$ and column $j$ in the matrix $A$; a square $s(i,j)$, with $1 \leq i \leq r-1$ and $ 1 \leq j \leq n-1$ will be said to be {\em black} if the product of the entries in $A$, $a_{i,j}, a_{i,j+1}, a_{i+1,j}, a_{i+1,j+1}$ is $-1$, and {\em white} otherwise (see Figure  \ref{chess}). 

\begin{figure}[h]
\begin{center}
 \includegraphics[width=0.9\textwidth]{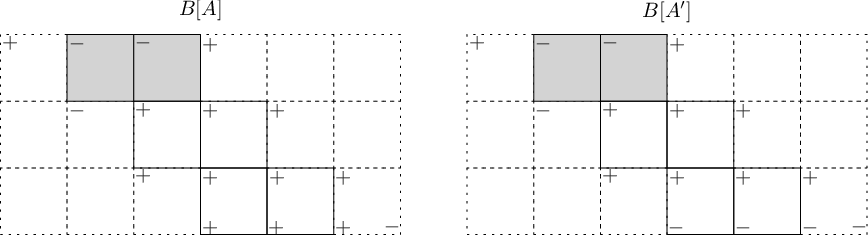}
\caption{Two matrices $A$ and $A'$ with their correspondent chessboards $B[A]$ and $B[A']$. The dotted squares do not play any role in the definition of ${\M}_A$ and ${\M}_{A'}$ (see Remark \ref{squares_involved}).} \label{chess}
\end{center}
\end{figure}

By definition, we notice that the chessboard $B[A]$ is invariant under reorientation of columns and rows.  It turns out that the chessboards represent the reorientation classes of LOMs as we will see in the next theorem.

\begin{remark}\label{entry-determined}
Let $s(i,j)$ be a square of a chessboard $B[A]$ and suppose that three of the four entries corresponding to $s(i,j)$ are known. Then, the unknown entry of $s(i,j)$ is completely determined by the color of $s(i,j)$. 
\end{remark}
  
  \begin{theorem}\label{chessboard=reo_classes}
Let $A,A'$ be two matrices. Then ${\M}_{A'}\in [{\M}_{A}]$ if and only if $B[A]=B[A']$.
\end{theorem}
\begin{proof}
   If ${\M}_{A'}\in [{\M}_{A}]$, then there exists a set of columns and rows of $A$ to obtain $A'$  by Remark \ref{col_y_filas}. As the reorientation of any set of columns and rows does not change
the chessboard we conclude that $B[A]=B[A']$.
Now suppose that $B[A]=B[A']$. Recall that the chessboard is invariant under reorientations of columns and rows. We will give a procedure to obtain the matrix $A'$ from  $A$ as follows. First, we reorient the necessary columns of $A$ to obtain the first row of $A'$. Next, if all the entries of some row $i\ge 1$ of $A$ match those in row $i$ of $A'$, we reorient (if necessary) row $i+1$ of $A$ in order to obtain $a'_{i+1,1}$. Then, notice that in fact we have already obtained all entries of row $i+1$ of $A'$ by Remark \ref{entry-determined} since $B[A]=B[A']$. Proceeding in this way, we may obtain matrix $A'$ concluding that ${\M}_{A'}\in [{\M}_{A}]$ and so the theorem holds.
\end{proof}


By Remark \ref{prop-lom} (i)  it follows that there are some squares in $B[A]$ that do not play any role in the definition of ${\M}_A$ (see Figure \ref{chess}).
 \begin{remark}\label{squares_involved}
  Let $A$ be a $(r\times n)$-matrix. Then the only squares of $B[A]$ that are involved in the construction of ${\M}_A$ are the squares
$s(i,j)$ for $i=1\ldots,r-1$ and $j=i+1,\ldots, n-r+i-1$.
\end{remark}

We are now ready to calculate the number of reorientation classes of rank $r$ LOMs on $n$ elements.

\begin{theorem}\label{number-of-reorientations}
    The number of reorientation classes of rank $r$ LOMs on $n$ elements is $$2^{(n-r-1)(r-1)}$$ 
\end{theorem}
\begin{proof}
 Let $\M$ be a rank $r$ LOM on $n$ elements and let $A$ be a $(r\times n)$-matrix such that $\M=\M_A$.  By Remark \ref{squares_involved} for each row $i=1,\ldots,r-1$ of $A$ there are $n-r-1$ squares that are involved in the construction of $\M_A$. Thus, there are $2^{(n-r-1)(r-1)}$ different ways to create a chessboard from a $(r\times n)$-matrix, which in turn are the number of reorientation classes of rank $r$ LOMs on $n$ elements by Theorem \ref{chessboard=reo_classes}.
\end{proof}



\subsection{$k$-Roudneff's conjecture for low ranks}\label{computer r=7_k=2}
In this section we prove $k$-Roudneff's Conjecture for LOMs for the cases $r=7$, $k=2$ and $n\ge 11$ (Corollary \ref{2-roudneff}), for $r=8$, $k=2,3$ and $n=11,12$ (Theorem \ref{thm_computer8}), for $r=9$, $k=2$ and $n=12$ and also for $r=9$, $k=3$ and $n\ge 12$ (Corollary \ref{3-roudneff}). 
We recall that \Cref{conjkRoudneff} was solved for LOMs when $k=1$ in  \cite{MR15}. 
\smallskip

In \cite{HKM24} it was proved that \Cref{conjkRoudneff} can be reduced to a finite number of cases as shown in the following theorem.

\begin{theorem}\cite[Theorem 3.16]{HKM24}\label{prop:onlythebaseishard} 
If $f_{\M'}(r',n',k)\leq c_{r'}(n',k)$ for every  rank $r'\leq r$ uniform oriented matroid $\M'$  on $n'=2(r'-k)+1$ elements, then $$f_{\M}(r,n,k)\leq c_{r}(n,k)$$ for every rank $r$ oriented matroid $\M$  on $n\geq 2(r-k)+1$ elements.
\end{theorem}

In fact, the above result can be applied only for LOMs as we will see below (Proposition \ref{coro-LOMs}). To prove this, first we need the following inequality.  

\begin{lemma}\cite[Lemma 3.15]{HKM24}\label{thm:onlythebaseishard} 
 Let $\mathcal{M}$ be  rank $r\ge2k+1\ge3$ uniform
 oriented matroid  and let $e\in E$. Then,
$$f_{\M}(r,n,k) \le f_{\M/ e}(r-1,n-1,k)+f_{\M\setminus e}(r,n-1,k).$$
\end{lemma}

The following observation is well known and will be used in the following proposition.
\begin{remark}\cite[Proposition 7.6.7]{BVSWZ99}\label{closedLOMs}
    The family of LOMs is closed under contraction and deletion of elements.
\end{remark}

Next, we show that Theorem \ref{prop:onlythebaseishard} can be apply only for LOMs. 

\begin{proposition}\label{coro-LOMs} 
If $f_{\M'}(r',n',k)\leq c_{r'}(n',k)$ for every  rank $r'\leq r$ LOM $\M'$  on $n'=2(r'-k)+1$ elements, then $$f_{\M}(r,n,k)\leq c_{r}(n,k)$$ for every rank $r$ LOM $\M$  on $n\geq 2(r-k)+1$ elements.
\end{proposition}
\begin{proof}
 Let us prove the result by induction on $r$ and $n$. By \Cref{thm:onlythebaseishard},
$f_{\M}(r,n,k) \le f_{\M/ e}(r-1,n-1,k)+f_{\M\setminus e}(r,n-1,k)$. Now, fix $r$ and let  $n>2(r-k)+1$. Notice that the inequality $f_{\M/ e}(r-1,n-1,k)\le  c_{r-1}(n-1,k)$ follows since $\M/ e$ is a LOM (Remark \ref{closedLOMs}) and by assumption all rank $r'=r-1$  LOM $\mathcal{M}'$ on $n'=2(r'-k)+1$ elements satisfy $f_{\M'}(r,n',k)\le c_{r}(n',k)$. On the other hand, the inequality 
 $f_{\M\setminus e}(r,n-1,k)\le c_r(n-1,k)$ then follows since $\M\setminus e$ is a  LOM (Remark \ref{closedLOMs}) and by induction on $n$ we know that it is verified for all rank $r$ LOM on $n-1$ elements.
Thus by induction this also holds for $n-1\ge 2(r-k)+1\ge 2(r'-k)+1$.
Now, a straight-forward computation using Equation (\ref{formula-ciclico}) yields
$$c_{r}(n-1,k)+c_{r-1}(n-1,k)=c_{r}(n,k).$$ Thus, we obtain that $f_{\M}(r,n,k) \le c_r(n,k)$.   
\end{proof}


\subsubsection{The computer program to obtain $f_{\mathcal{M}}(r,n,k)$.}

In \cite[Subsection 4.1]{HKM24}, the authors provided a computer program to obtain $f_\mathcal{M}(r,n,k)$ from its chirotope, available at \cite{supplemental_data} (in fact they obtained it via something called the $o$-vector whose $i$-th entry is $o_i$ with $i=0,\ldots,\lfloor\frac{r-1}{2}\rfloor$ and $o_i$ is the number of $i$-neighborly reorientations of $\mathcal{M}$ that are not $i+1$-neighborly). Now, for a LOM $\mathcal{M}_A$, we may use Equation (\ref{bases-lom}) to obtain the signature of its basis an so its chirotope. Thus, using the computer program mentioned above we obtain  $f_{\mathcal{M}_A}(r,n,k)$.
\smallskip

By Remark \ref{remarkk-Roudff} it is enough to check Conjecture \ref{conjkRoudneff} only for the reorientation classes $[\M]$ of $\M$. So, in the case of LOMs it is sufficient  by Theorem \ref{chessboard=reo_classes} to choose a representative matrix of a chessboard to obtain a representative matrix of its corresponding reorientation class. Hence, the computer program (available at \cite{supplemental_data}) has to compute $2^{(n-r-1)(r-1)}$ different chirotopes to obtain $f_{\mathcal{M}_A}(r,n,k)$ for the case of LOMs by Theorem \ref{number-of-reorientations}.



\subsubsection{$k$-Roudneff's conjecture for LOMs with $r=7$, $k=2$ and $n\geq 11$}
  
We first notice that \Cref{conjkRoudneff} was already proved for $r=7$, $k=2$ and $n=10$ in \cite[Theorem 4.1 (d)]{HKM24}. 
 Now, in order to prove the conjecture for LOMs with $r=7$, $k=2$ and $n\geq 11$ it is 
  enough to prove only the case $n=11$ by Proposition  \ref{coro-LOMs} since for $r<7$ and $k=2$  \Cref{conjkRoudneff} was proved in \cite{HKM24}.
 As mentioned above, we have to analyze $2^{(n-r-1)(r-1)}=2^{18}=262$ $144$ different matrices. We implemented our computer program and we resume the results in the following theorem. 

\begin{theorem}\label{thm_computer7}
 Let $\mathcal{M}\not\in [C_7(11)]$ be a rank $7$ LOM on $11$ elements, then $f_{\M}(7,11,2)\leq c_{7}(11,2)=112$ and there are exactly 255 reorientation classes with $f_{\M}(7,11,2)= c_{7}(11,2)$.
\end{theorem}

\begin{corollary}\label{2-roudneff}
    Let $\mathcal{M}$ be a LOM of rank $r= 7$ on $n\geq 11$ elements, then $$f_{\mathcal{M}}(7,n,2)\le c_7(n,2).$$
\end{corollary}

On the other hand, notice that \Cref{conjkRoudneff} holds for $r=7$, $k=3$ since it was already proved for any oriented matroid in \cite[Theorem 3.17]{HKM24}.

\subsubsection{$k$-Roudneff's conjecture for LOMs with $r=8$, $k=2,3$ and $n=11,12$}

By Theorem \ref{number-of-reorientations} we have to analyze $2^{14}=16$ $384$ and $2^{21}=2$ $097$ $152$ different matrices for the cases $n=11$ and $n=12$, respectively. The results are presented in the following theorem. 

\begin{theorem}\label{thm_computer8}
 Let $\mathcal{M}\not\in [C_8(n)]$ be a rank $8$ LOM on $n$ elements, then the following hold:
 \begin{itemize}
    \item [(a)]  If $n=11$, then $f_{\M}(8,11,2)\leq c_{8}(11,2)$ and there are exactly 255 reorientation classes with $f_{\M}(8,11,2)= c_{8}(11,2)$.
     \item [(b)]  If $n=11$, then $f_{\M}(8,11,3)\leq c_{8}(11,3)$ and there are exactly 251 reorientation classes with $f_{\M}(8,11,3)= c_{8}(11,3)$.
    \item [(c)]  If $n=12$, then $f_{\M}(8,12,2)\leq c_{8}(12,2)$ and there are exactly 511 reorientation classes with $f_{\M}(8,12,2)= c_{8}(12,2)$.

    \item [(d)] If $n=12$, then $f_{\M}(8,12,3)\leq c_{8}(12,3)$ and there are exactly 511 reorientation classes with $f_{\M}(8,12,3)= c_{8}(12,3)$.
    
 \end{itemize}
\end{theorem}

On the other hand, we notice that the conjecture holds for LOMs with $r=8$, $k=3$ and $n\ge 20$ by Theorem \ref{rparkmax}. We also observe that in order to prove the conjecture for LOMs with $r=8$, $k=2$ and $n\ge 13$ (using Proposition \ref{coro-LOMs}) the computer program has to calculate $2^{28}=268$ $435$ $456$ different matrices by Theorem \ref{number-of-reorientations}.


\subsubsection{$k$-Roudneff's conjecture for LOMs with $r=9$, $k=3$ and $n\geq 12$ and for $r=9$, $k=2$ and $n=12$.}


To prove \Cref{conjkRoudneff} for LOMs with $r=9$, $k=3$ and $n\geq 13$,  it is enough by Proposition \ref{coro-LOMs} to prove it for LOMs with  $(r,n,k)\in \{(7,9,3),(8,11,3),(9,13,3)\}$ (notice that for $k=3$, $r\ge 2k+1=7$). So, the case $(7,9,3)$ is the case $n=r+2$ and holds by Remark \ref{remarkk-Roudff} and the case $(8,11,3)$ was proved in Theorem \ref{thm_computer8} (b). For the cases $r=9$ and $n=12$ and $13$, we have to analyze $2^{16}= 65$ $536$ and $2^{24}=16$ $777$ $216$ different matrices  by Theorem \ref{number-of-reorientations}, respectively. 
We implemented our computer program and we summarize the results in the following theorem. 

\begin{theorem}\label{thm_computer9}
 Let $\mathcal{M}\not\in [C_9(n)]$ be a rank $9$ LOM on $n$ elements, then the following hold:
 \begin{itemize}

\item [(a)]  If $n=12$, then $f_{\M}(9,12,2)\leq c_{9}(12,2)$ and there are exactly 511 reorientation classes with $f_{\M}(9,12,2)= c_{9}(12,2)$.
    
    \item [(b)]  If $n=12$, then $f_{\M}(9,12,3)\leq c_{9}(12,3)$ and there are exactly 511 reorientation classes with $f_{\M}(9,12,3)= c_{9}(12,3)$.
    
    \item [(c)]  If $n=13$, then $f_{\M}(9,13,3)\leq c_{9}(13,3)$ and there are exactly 1023 reorientation classes with $f_{\M}(9,13,3)= c_{9}(13,3)$.
 \end{itemize}
\end{theorem}

\begin{corollary}\label{3-roudneff}
    Let $\mathcal{M}$ be a LOM of rank $r= 9$ on $n\geq 12$ elements, then $$f_{\mathcal{M}}(9,n,3)\le c_9(n,3).$$
\end{corollary}


On the other hand, notice that \Cref{conjkRoudneff} holds for $r=9,k=4$ since it was already proved for any oriented matroid in \cite[Theorem 3.17]{HKM24}.  Finally, we observe that in order to prove the conjecture for LOMs with $r=9$ $k=2$ and $n\ge 15$ (using Proposition \ref{coro-LOMs}) the computer program has to calculate $2^{40}$ different matrices by Theorem \ref{number-of-reorientations}.

\section{$k$-Roudneff's conjecture holds asymptotically}
\label{asymptotics}

The main result of this section is to prove $k$-Roudneff's conjecture asymptotically for fixed $k\ge 1$ and $r\ge 2k+2\ge 4$, when $n\rightarrow \infty$. Denote by $f(r,n,k)$ the maximum number of $k$-neighborly reorientations over all rank $r$ oriented matroids on $n$ elements. By Lemma \ref{thm:onlythebaseishard}, if $\M$ is such that $f_{\M}(r,n,k)=f(r,n,k)$, then it follows that $f(r,n,k) \le f_{\M/ e}(r-1,n-1,k)+f_{\M\setminus e}(r,n-1,k)\le f(r-1,n-1,k)+f(r,n-1,k)$. The following recursive upper bound will be very useful in this section. 

\begin{corollary}\label{recursive_formula}
    $f(r,n,k) \le f(r-1,n-1,k)+f(r,n-1,k).$
\end{corollary}

For $n\ge r\geq 2k+2\ge 4$, we define $$F(r,n,k)=\dfrac{2\Biggl((n-r)+\dbinom{r}{k+1}+2^{r-1}\Biggl)}{(r-1-2k)!}^{r-1-2k}$$

We will prove in Theorem \ref{thm_asymptotics} that $f(r,n,k) \le F(r,n,k)$ for every $r\ge 2k+2\ge 4$ and $n\ge r+1\ge 5$ but first, we need some previous results. $k$-Roudneff's conjecture was proved for $r = 2k + 1\ge3$ and $n \ge r+2$ in \cite[Corollary 3.17]{HKM24}. On the other hand, the value of $c_r(n,k)$ was computed for $n\ge r+1$ when $r=2k+1$ in \cite[Theorem 3.11 and Proposition 3.12]{HKM24}. Combining these results and Remark \ref{remarkk-Roudff} we obtain the following result that will be  useful in this section.

\begin{theorem}\cite{HKM24}\label{r_odd_max_ort}
  Let $\mathcal{M}$ be a rank $r=2k+1\ge 3$ oriented matroid on $n\ge r+1$ elements, 
  then  
  \begin{equation*}
     f_{\mathcal{M}}(r,n,k)\le\left\{
	       \begin{array}{ll}
		 c_r(n,k)=2       & \text{ if    $n\ge r+2$; }  \\ \ \\
		  c_r(n,k)={r+1 \choose k+1}      & \text{ if    $n=r+1$. }\\

	       \end{array}
	     \right.
\end{equation*}
\end{theorem}

The next observation can be deduced from the binomial theorem.
\begin{remark}\label{ref1}
Given positive integes $x$ and $y$ it holds that $\dfrac{x^{y-1}}{(y-1)!}+\dfrac{x^y}{y!}< \dfrac{(x+2)^y}{y!}$
\end{remark}

Next, we obtain the following recursive lower bound of $F(r,n,k)$ that will be useful in order to prove Theorem \ref{thm_asymptotics}.

\begin{remark}\label{desigualdad F}
Let $r\ge 2k+2\ge 4$ and $n\geq r+1\ge 5$, then
$$F(r-1,n-1,k)+ F(r,n-1,k)< F(r,n,k)$$
\end{remark}

\begin{proof}
{\tiny
$$
F(r-1,n-1,k)+ F(r,n-1,k)
\begin{array}[t]{l}
=\displaystyle  \dfrac{2\Biggl((n-r)+ \dbinom{r-1}{k+1}+2^{r-2}\Biggl)}{(r-2-2k)!}^{r-2-2k} + \dfrac{2\Biggl((n-r-1)+ \dbinom{r}{k+1}+2^{r-1}\Biggl)}{(r-1-2k)!}^{r-1-2k}\\ \ \\
<\displaystyle  \dfrac{2\Biggl((n-r-1)+ \dbinom{r}{k+1}+2^{r-1}\Biggl)}{(r-2-2k)!}^{r-2-2k} + \dfrac{2\Biggl((n-r-1)+ \dbinom{r}{k+1}+2^{r-1}\Biggl)}{(r-1-2k)!}^{r-1-2k}\\ \ \\
<\displaystyle  \dfrac{2\Biggl((n-r)+ \dbinom{r}{k+1}+2^{r-1}\Biggl)}{(r-1-2k)!}^{r-1-2k}=F(n,r,k)
\end{array}
$$
}


where the first inequality holds since ${r-1\choose k+1}<{r\choose k+1}$ for $r\ge 2k+2\ge 4$ and the last inequality follows by Remark \ref{ref1}.
\end{proof}

The value of  $f(r,r,k)$ makes sense if we consider the oriented matroid as an arrangement. As we pointed out in the introduction (Subsection \ref{sub-roudneff}), an oriented matroid on $n$ elements of rank $r$ is naturally associated with some arrangement of $n$ (pseudo) hyperplanes in the projective space of dimension $d=r-1$. The arrangement decomposes the projective space  into $d$-cells which in turn correspond to the acyclic reorientations (i.e., $0$-neighborly reorientations) of the oriented matroid (We refer the reader to \cite{FL78} to more details on the Topological Representation Theorem). In this sense, it is well-known  that $f(r,r,0)\le 2^{r-1}$   (see \cite{Cor80}). 

\begin{remark}\label{n=r}
    Let $r\ge 2k+2\ge 4$, then
    $f(r,r,k)\le 2^{r-1}< F(r,r,k)$.
\end{remark}

\begin{proof}
First notice that $f(n,r,k)\le f(n,r,k')$ for $k'\le k$. Second, as we mention above, $f(r,r,0)\le 2^{r-1}$. Hence, we obtain that $f(r,r,k)\le 2^{r-1}$. The second inequality holds directly by definition of the function $F$.
\end{proof}

\begin{remark} \label{2k+2_2k+3}
Let $k\ge 1$, then
    $f(2k+2,2k+3,k) < F(2k+2,2k+3,k).$
\end{remark}
\begin{proof}
By Corollary \ref{recursive_formula}, we have that  $f(2k+2,2k+3,k) \le f(2k+1,2k+2,k)+f(2k+2,2k+2,k)\le {2k+2 \choose k+1}+ 2^{2k+1}$, where the last inequality holds by Theorem \ref{r_odd_max_ort} and Remark \ref{n=r}. Thus, the result follows since $F(2k+2,2k+3,k)=2+2{2k+2 \choose k+1}+2^{2k+2}.$ 
\end{proof}

The next two lemmas prove the induction basis applying in the proof of Theorem \ref{thm_asymptotics}.

\begin{lemma} \label{r+1}
Let $r\ge 2k+2\ge 4$, then
    $f(r,r+1,k)< F(r,r+1,k).$
\end{lemma}
\begin{proof}
 We will prove the lemma by induction on $r$.  The result follows for $r=2k+2$ by Remark \ref{2k+2_2k+3}. Now, assume by induction that the lemma holds for $r-1\ge 2k+2$. By Corollary \ref{recursive_formula}, $f(r,r+1,k) \le f(r-1,r,k)+f(r,r,k)< F(r-1,r,k)+ F(r,r,k)$, where the last inequality holds by induction hypothesis and Remark \ref{n=r}. Therefore, $f(r,r+1,k)< F(r,r+1,k)$ by Remark \ref{desigualdad F}.
\end{proof}

\begin{lemma}\label{hi1}
Let  $n\geq 2k+3\ge 5$, then
$f(2k+2,n,k)\leq F(2k+2,n,k)$.
\end{lemma}
\begin{proof}
The result follows for $n=2k+3$ by Remark \ref{2k+2_2k+3}. Now, assume by induction that the result holds for $n-1\ge 2k+3$. By Corollary \ref{recursive_formula}, $f(2k+2,n,k) \le f(2k+1,n-1,k)+f(2k+2,n-1,k)\le 2+ F(2k+2,n-1,k)=2(n-2k-2)+ {2k+2\choose k+1}+2^{2k+1}=F(2k+2,n,k)$, where the second inequality holds by Theorem \ref{r_odd_max_ort} and induction hypothesis, concluding the proof.
\end{proof}

We are now ready to prove the next theorem.
\begin{theorem}\label{thm_asymptotics}
Let $r\geq 2k+2\ge 4$ and  $n\geq r+1\ge 5$, then
\begin{center}
$f(r,n,k)\leq F(r,n,k)$
\end{center}
\end{theorem}
\begin{proof}
We shall prove the theorem by induction on $r$ and $n$. On the one hand, the result follows for $r=2k+2\geq 4$ and every $n\ge r+1\geq 5$  by Lemma \ref{hi1}. On the other hand, the result follows for  $n=r+1\ge 5$ and every $r\ge 2k+2\ge 4$ by Lemma \ref{r+1}. So, assume by induction that the result is true for $r-1\ge 2k+2$ and that for $r$ the result holds for up to $n-1\ge r+1$.

By Corollary \ref{recursive_formula} we obtain that $f(r,n,k)\leq f(r-1,n-1,k)+f(r,n-1,k)$. Now, since $f(r-1,n-1,k)\le F(r-1,n-1,k)$  and $f(r,n-1,k)\le F(r,n-1,k)$ by the induction  hypotheses  on $r$ and $n$, respectively, we conclude that $$f(r,n,k)\leq F(r-1,n-1,k)+F(r,n-1,k)<F(r,n,k),$$
where the last inequality is due to Remark \ref{desigualdad F}, concluding the proof.
\end{proof}
\medskip




Finally,  notice that for fixed $r$ and $k$, we have that $F(r,n,k)\le c_r(n,k)$  if $n\rightarrow \infty$  since $$F(r,n,k)\rightarrow \dfrac{2n^{r-1-2k}}{(r-1-2k)!} \rightarrow 2{n \choose r-1-2k}$$ and  $$c_r(n,k)=2\sum\limits_{i=0}^{r-1-2k} \binom{n-1}{i}\rightarrow 2{n \choose r-1-2k}+ O(2n^{r-1-2k})$$ as $n\rightarrow \infty$. 
\\

Hence,   Conjecture \ref{conjkRoudneff} holds asymptotically  when $n\rightarrow \infty$ by Theorem \ref{thm_asymptotics}.

\begin{corollary}\label{asintotico_k-Roudf}
For fixed $k\ge 1$ and $r\geq 2k+2\ge 4$, 
    $$f(r,n,k)\le 2{n \choose r-1-2k}+ O(2n^{r-1-2k})$$
as $n\rightarrow \infty$.
\end{corollary}

\section{Conclusions}\label{Conclusions}

This article is a contribution towards $k$-Roudneff's conjecture (Conjecture \ref{conjkRoudneff}). In Section \ref{LOM_k-Roudf} we prove it for LOMs with even rank $r=2k+2\ge 4$ and $n\geq 3r-4$ (Theorem \ref{rparkmax}) and for some low ranks with the aid of the computer, the most important computational results being the cases $r=7$, $k=2$, $n\ge 11$ (Corollary \ref{2-roudneff}) and  $r=9$, $k=3$ and $n\ge 12$ (Corollary \ref{3-roudneff}).  
We also provide a general upper bound of $f_{\M}(r,n,k)$ when $\M$ is a LOM  (Theorem \ref{theo-cota-sup}).  In Section \ref{asymptotics} we prove that for fixed $r\ge 4$ and $k\ge 1$ and any 
oriented matroid on $n$ elements, $k$-Roudneff's conjecture holds asymptotically as $n\rightarrow \infty$ and thus giving more credit to the conjecture (Corollary \ref{asintotico_k-Roudf}).


\subsubsection*{Acknowledgements.}
 R. Hern\'andez-Ortiz and L.\ P.\ Montejano were supported by SGR Grant 2021 00115 and by the project HERMES, funded by INCIBE and the European Union NextGenerationEU/PRTR.
L. P Montejano was supported by the Ministry of Education and Science, Spain, under the project PID2023-146925OB-I00 and by the proyect ACoGe: PID2022-137283NB-C22.
\bibliographystyle{my-siam}
\bibliography{lit}

\end{document}